\documentclass[11pt]{amsart}
\usepackage{graphicx}
\usepackage[pagebackref, colorlinks=true, linkcolor=black, citecolor=black, urlcolor=black]{hyperref}

\usepackage{xargs}                      
\usepackage[colorinlistoftodos,prependcaption 
]{todonotes}                                                                                                                 

\usepackage{comment}
\usepackage{calc}

\usepackage{cite}
\makeatletter
\newcommand{\citecomment}[2][]{\citen{#2}#1\citevar}
\newcommand{\citeone}[1]{\citecomment{#1}}
\newcommand{\citetwo}[2][]{\citecomment[,~#1]{#2}}
\newcommand{\citevar}{\@ifnextchar\bgroup{;~\citeone}{\@ifnextchar[{;~\citetwo}{]}}}
\newcommand{\citefirst}{\@ifnextchar\bgroup{\citeone}{\@ifnextchar[{\citetwo}{]}}}
\newcommand{\cites}{[\citefirst}
\makeatother

\usepackage{amsmath, amsthm, amsfonts}
\usepackage{enumerate}
\usepackage{bbm}
\usepackage{mathrsfs}
\usepackage{amssymb}
\usepackage{mathtools}
\usepackage[all]{xy}
\usepackage{tikz}
\usepackage{tikz-cd}
\usetikzlibrary{matrix,arrows,decorations.pathmorphing}
\usepackage{xcolor}
\usepackage{bm}
	\usetikzlibrary{calc}
	\usetikzlibrary{trees}
	\tikzstyle{every picture}=[scale=.35,inner sep=0]
	\usepackage{color}

\newtheorem{thm}{Theorem}

\newtheorem{cor}{Corollary}

\newtheorem{theorem}{Theorem}[section]
\newtheorem{lemma}[theorem]{Lemma}
\newtheorem{proposition}[theorem]{Proposition}
\newtheorem{corollary}[theorem]{Corollary}

\theoremstyle{definition}

\theoremstyle{remark}

\numberwithin{equation}{section}

\newcommand{\M}{\mathcal{M}}

\newcommand{\Q}{\mathbb{Q}}
\newcommand{\Z}{\mathbb{Z}}

\newcommand{\ra}{\rightarrow}
\renewcommand{\P}{\mathbb{P}}

\newcommand{\E}{\mathbb{E}}

\newcommand{\calX}{\mathcal{X}}
\renewcommand{\H}{\mathbb{H}}

\newcommand{\W}{\mathbb{W}}

\newcommand{\ov}{\overline}

\newcommand{\Hg}{\P\E^k_{g}}

\makeatletter
\@namedef{subjclassname@2020}{\textup{MSC2020}}
\makeatother

\begin{document}

\title[$k$-canonical divisors through Brill-Noether special points]{$k$-canonical divisors through\\ Brill-Noether special points}

\author[I.~Gheorghita]{Iulia Gheorghita}
\address{Iulia Gheorghita  
\newline \indent Department of Mathematics
\newline \indent Boston College, Chestnut Hill, MA 02467}
\email{gheorgiu@bc.edu}

\author[N.~Tarasca]{Nicola Tarasca}
\address{Nicola Tarasca 
\newline \indent Department of Mathematics \& Applied Mathematics
\newline \indent Virginia Commonwealth University, Richmond, VA 23284}
\email{tarascan@vcu.edu}

\subjclass[2020]{14H10, 14H51, 14C25 (primary), 30F30 (secondary)}
\keywords{Moduli spaces of algebraic curves with differentials, Weierstrass  points, Brill-Noether special points, pseudo-effective cone of divisors, Teichm\"uller curves}

\begin{abstract} 
Inside the projectivized $k$-th Hodge bundle, we construct a collection of  divisors obtained by imposing vanishing at a Brill-Noether special point. We compute the classes of the closures of such divisors in two ways, using incidence geometry and restrictions to various families, including pencils of curves on K3 surfaces and pencils of Du Val curves.
We also show the extremality and rigidity of the closure of the incidence divisor consisting of smooth pointed curves together with 
a canonical or $2$-canonical divisor passing through the marked point.
\end{abstract}

\vspace*{-0.5pc}

\maketitle

\vspace{-1.5pc}

\section{Introduction}

Applying ideas from Brill-Noether theory, we 
construct  effective divisors in the moduli space of \mbox{$k$-canonical} divisors on algebraic curves.
To describe the ambient moduli space in question, start with
the $k$-th Hodge bundle $\E^k_g$. This is the vector bundle of stable $k$-differentials over the moduli space $\ov{\M}_g$ of stable curves of genus $g$.
It is defined as $\E^k_g:=\pi_* \left( \omega_\pi^{\otimes k}\right)$, where 
 $\pi\colon \mathcal{C}_g\rightarrow \ov{\M}_g$ is the universal curve with  relative dualizing sheaf $\omega_\pi$.
The projectivization $\P\E^k_g$ of $\E^k_g$ compactifies the moduli space of $k$-canonical divisors on smooth algebraic curves.
Various divisor classes on $\P\E^k_g$ have recently been considered in the literature, including the class of the closure of the divisorial stratum in $\P\E^k_g$ consisting of $k$-differentials with a double zero, and the first Chern class of Prym-Tyurin vector bundles 
 \cite{korotkin2011tau, korotkin2013tau, sauvagetcohomology, korotkin2019tau}.

Here our focus  is on divisors consisting of $k$-differentials which vanish at a Brill-Noether special point. To start, consider the effective divisor
\[ 
\W^k_g := \left\{(C, \mu) \in \Hg \; \big| \; C \text{ is smooth and } \mu  \text{ vanishes at a Weierstrass point} \right\}\!. 
\] 
We compute the class of $\W^k_g$ in terms of  standard generators of the Picard group of $\Hg$ with rational coefficients, generalizing the case $k=1$ treated in \cite{gheorghita2018effective}:

\begin{thm}
\label{thm:kW}
For $k\geq 1$ and $g\geq 2$, one has
\begin{multline*}
\left[\ov{\W}^k_g\right] = -g(g^2-1)\,\eta +  k(6g^2+4g+2)\,\lambda - k{g+1 \choose 2}\,\delta_0 \\
 - \sum_{i=1}^{\lfloor g/2 \rfloor} k(g+3)i(g-i)\,\delta_i \,\,\in\,\,\textup{Pic}_\Q\left(\Hg\right).
\end{multline*}
\end{thm}

Here, $\eta := c_1\big(\mathscr{O}_{\P\E^k_g}(-1)\big)$;  $\lambda$ is the pull-back from $\ov{\M}_g$ of the first Chern class of the Hodge bundle; $\delta_0$ is the pull-back from $\ov{\M}_g$ of the class of the locus of curves whose general element has a non-disconnecting node; 
 and $\delta_i$ is the pull-back from $\ov{\M}_g$ of the class of the locus of curves whose general element has a disconnecting node and components of genera $i$ and $g-i$, for $1\leq i\leq \lfloor g/2 \rfloor$.

Additionally, we consider  divisors in $\Hg$ obtained by imposing vanishing at an arbitrary Brill-Noether special point. 
We briefly review the setup and refer the reader to \S\ref{sec:BNdiv} for the required background.
For $g\geq 2$ and a sequence 
\[
\bm{a}:0\leq a_0< \dots <a_r\leq d, 
\]
define $\rho(g,r,d,\bm{a}):= g-(r+1)(g-d+r)-\sum_{i=0}^r (a_i-i)$. When $\rho(g,r,d,\bm{a})=-1$, 
a general curve of genus $g$ contains only finitely many points $P$
where a linear series $\ell$ of type $\mathfrak{g}^r_d$ has vanishing sequence $\bm{a}^\ell(P)\geq \bm{a}$.
In this case, the locus $\M_{g,d}^{\bm{a}}$ of pointed curves $(C,P)$ which admit a linear series $\ell$ of type $\mathfrak{g}^r_d$ with $\bm{a}^\ell(P)\geq \bm{a}$ is a proper subvariety with a unique irreducible divisorial component inside the moduli space $\M_{g,1}$ of pointed genus-$g$ curves  \cite{EH-1}. For $d=2g-2$, $r=g-1$, and $\bm{a}=(0,1,2,\dots,g-2,g)$, one obtains the divisor of curves with a marked Weierstrass point. The class of the divisorial component of $\ov{\M}_{g,d}^{\bm{a}}$ lies in the cone spanned by the pullback of the Brill-Noether divisor class $\mathcal{BN}_g$ from $\ov{\M}_g$ and the  class $\mathcal{W}_g$ of the divisor of marked Weierstrass points \cite{EH-1}. Specifically, the divisorial component of $\ov{\M}_{g,d}^{\bm{a}}$ has class equal to $\mu_{g,d,\bm{a}} \, \mathcal{BN}_g + \nu_{g,d,\bm{a}} \, \mathcal{W}_g$, for some nonnegative rational coefficients $\mu_{g,d,\bm{a}}$ and $\nu_{g,d,\bm{a}}$ computed in \cite{FT}. These classes  have been used in various settings, including the study of the birational geometry of moduli spaces of pointed curves \cite{MR1953519, MR2530855}, and the proof of the non-varying property of sums of Lyapunov exponents for certain strata of abelian and quadratic differentials in low genus \cite{MR3033521, chen2014quadratic}. 

The loci $\M_{g,d}^{\bm{a}}$ induce natural subvarieties of $\P\E^k_g$.
Namely, given $g\geq 2$ and a sequence $\bm{a}:0\leq a_0< \dots <a_r\leq d$,
define the locus ${\H}^{\bm{a}}_{g,d}$ in $\Hg$ as 
\[
{\H}^{\bm{a}}_{g,d}:=\left\{ (C, \mu)\in \Hg \,\Bigg|\, 
\begin{array}{l}
\mbox{$C$ is smooth and $\mu$ vanishes at some point $P$}\\[0.2cm]
\mbox{such that }\bm{a}^\ell(P) \geq \bm{a} \mbox{ for some $\ell \in G^r_d(C)$}
\end{array}
\right\}.
\]
When $\rho(g,r,d,\bm{a})=-1$, the locus ${\H}^{\bm{a}}_{g,d}$ has a divisorial component.
For $d=2g-2$, $r=g-1$, and $\bm{a}=(0,1,2,\dots,g-2,g)$, the locus ${\H}^{\bm{a}}_{g,d}$ specializes to the above divisor $\W^k_g$.
We show that the class of the closure of the divisorial component of ${\H}^{\bm{a}}_{g,d}$, denoted by $\left[ \ov{\H}^{\bm{a}}_{g,d} \right] \in \textup{Pic}_\Q\left(\Hg\right)$, lies in the cone spanned by the pullback of the Brill-Noether divisor class $\mathcal{BN}_g$ from $\ov{\M}_g$ and the divisor class $\left[\ov{\W}^k_g\right]$ from Theorem \ref{thm:kW}, analogously to the result from \cite{EH-1} for the divisorial components of $\M_{g,d}^{\bm{a}}\subset{\M}_{g,1}$. Specifically, we have:

\begin{cor}
\label{cor:kBN}
For $k\geq 1$, $g \geq 2$,  and   
 $\bm{a}: 0 \leq a_0 < \dots < a_r \leq d$  such that $\rho(g,r,d,\bm{a}) =-1$, 
one has
\[ 
\left[ \ov{\H}^{\bm{a}}_{g,d} \right] = 2k(g-1)\,\mu_{g,d,\bm{a}} \, \mathcal{BN}_g + \nu_{g,d,\bm{a}}  \left[\ov{\W}^k_g\right] \,\,\in\,\,\textup{Pic}_\Q\left(\Hg\right).
\]
\end{cor}

Next, we show how Theorem \ref{thm:kW} for $g=2$ complements a result from \cite{korotkin2019tau}.
The closure of the divisorial stratum in $\P\E^k_g$ of $k$-differentials with a double zero is computed for all $(g, k) \neq (2,2)$ in \cite{korotkin2011tau, korotkin2013tau, sauvagetcohomology, korotkin2019tau}.
For $g=k=2$, the formula from \cite{korotkin2019tau} specializes to a weighted sum of two components. Namely,
one has
\[
\left[\ov{\H}^2_2(2, 1, 1)\right]+ 2\left[\ov{\H}^2_2(2,2)\right] = 72 \lambda -10\eta -6\delta_0-6\delta_1
\] 
from  \cite[Def.~1.3, Thm 1.12]{korotkin2019tau}, where ${\H}^2_2(2, 1, 1)\subset \P\E^2_2$ is the divisorial stratum of quadratic differentials vanishing at a Weierstrass point, and ${\H}^2_2(2,2)\subset \P\E^2_2$ is the divisorial stratum consisting of squares of holomorphic differentials.
Since one has $\mathbb{W}^2_2 =  {\H}_2^2(2,1,1)$, we deduce:

\begin{cor}
The closure of the divisorial stratum ${\H}^2_2(2,2)\subset \P\E^2_2$ parametrizing squares of holomorphic differentials has class 
\[
\left[\ov{\H}^2_2(2,2)\right] = -2\eta + 12\lambda - \delta_0 \,\, \in\,\,  \mathrm{Pic}\left(\P\E^2_{2}\right).
\]
\end{cor}

We present two independent proofs of Theorem \ref{thm:kW} and Corollary \ref{cor:kBN}. For the first proof,
the key idea  is to regard  the locus ${\H}^{\bm{a}}_{g,d}$ as the projection of a locus in the space $\P\E^k_{g,1}$ obtained by pulling-back $\P\E^k_g$ via $\ov{\M}_{g,1}\rightarrow \ov{\M}_g$. Let
\begin{equation}
\label{eq:incdiv}
\mathbb{H}^k_{g, 1} \subset \P\E^k_{g,1} 
\end{equation}
be the \textit{incidence divisor} consisting of smooth pointed curves together with the class of a stable $k$-differential vanishing at the marked point.
In Proposition \ref{prop:BNirreducible}, we show that the divisorial component of $\ov{\H}^{\bm{a}}_{g,d}$ is the push-forward of the intersection of $\ov{\mathbb{H}}^k_{g, 1}$ and the pull-back of the divisorial component of $\ov{\M}_{g,d}^{\bm{a}}\subset\ov{\M}_{g,1}$, and thus deduce the divisorial class of  $\ov{\H}^{\bm{a}}_{g,d}$ in \S\ref{sec:proofThmskWkBN}.

An alternative approach is pursued in \S\ref{sec:testfamilies}, where the class of the divisorial component of $\ov{\H}^{\bm{a}}_{g,d}$ is computed by intersecting with various test families.
This proof has the advantage of showing explicit  restrictions of the locus $\ov{\H}^{\bm{a}}_{g,d}$, including restrictions to pencils of curves on K3 surfaces following \cite{MR852158, cukierman1993curves, MR2123229} and pencils of Du Val curves following \cite{arbarello2016explicit}.

Finally, we obtain the following  result on the incidence divisor.
Using a dense collection of Teichm\"uller curves of abelian and quadratic differentials
having negative intersection with the incidence divisor, we show:

\begin{thm}
\label{thm:Z1extremal}
For $k\in \{1,2\}$, the class of $\ov{\mathbb{H}}^k_{g, 1}$ is rigid and extremal in $\ov{\mathrm{Eff}}^1\left( \P\E^k_{g,1}\right)$.
\end{thm}

To review the notation:
A  class $E$ in the  cone $\ov{\mathrm{Eff}}^1( X)$ of pseudo-effective divisor classes  on a projective variety $X$ is called \textit{extremal} if $E = E_1 +E_2$ for $E_1$ and $E_2$  in $\ov{\mathrm{Eff}}^1( X)$ implies that both $E_1$ and $E_2$ are proportional to~$E$. An effective cycle class $E$ is called \textit{rigid} if any effective cycle with class $mE$ is supported on the support of $E$.

It is natural to ask whether the class of $\ov{\mathbb{H}}^k_{g, 1}$ is rigid and extremal also for $k\geq 3$. More generally, one can consider higher codimensional incidence varieties consisting of smooth $n$-pointed curves together with the class of a stable $k$-differential vanishing at all marked points. We have recently started the study of the classes of such varieties  in \cite{gt}. It is natural to ask whether the rigidity and extremality of the incidence divisor for $k\in \{1,2\}$ extend to its higher codimensional counterparts.

\subsection*{Acknowledgements} 
We would like to thank Dawei Chen for  helpful conversations on  $k$-differentials and the incidence variety compactification.


\section{Background on pointed Brill-Noether divisors}
\label{sec:BNdiv} 

Here we review the background on pointed Brill-Noether theory required in  \S\S\ref{sec:proofThmskWkBN}-\ref{sec:testfamilies} following \cite{MR910206, EH-1, FT}. 

For a smooth algebraic curve $C$, a linear series of type $\mathfrak{g}^r_d$ on $C$ is a pair $(L,V)$ where $L\in \mathrm{Pic}^d(C)$ and $V\subseteq H^0(C,L)$ is a subspace of dimension $r+1$.
The variety $G^r_d(C)$ parametrizes linear series of type $\mathfrak{g}^r_d$ on $C$.
For $\ell = (L,V)$ in $G^r_d(C)$, the \textit{vanishing sequence} of $\ell$ at a point $P$ in $C$ 
\[
\bm{a}^\ell(P): 0 \leq a_0 < \dots < a_r \leq d
\] 
is defined as the increasing sequence of vanishing orders of sections in $V$ at~$P$. 
For a sequence $\bm{a}: 0 \leq a_0 < \dots < a_r \leq d$, the \textit{adjusted Brill-Noether number}  is defined as 
\[
\rho\left(g,r,d,\bm{a}\right) :=g - (r+1)(g-d+r) - \sum_{i=0}^r \left(a_i - i\right).
\] 
The pointed version of the Brill-Noether Theorem \cite{MR910206} states that a general pointed curve $(C,P)$ of genus $g>0$ admits a linear series $\ell \in G^r_d(C)$ with vanishing sequence $\bm{a}^\ell(P) = \bm{a}$ if and only if 
\[\sum_{i=0}^r(a_i - i + g - d + r)_+ \leq g,\] 
where $(n)_+ := \max\{n, 0\}$ for $n \in \Z$. 
This implies \mbox{$\rho(g,r,d,\bm{a}^\ell(P))\geq 0$} for a general $(C,P)$ and any $\ell\in G^r_d(C)$. We refer to \cite{FT2} for explicit examples of smooth pointed curves satisfying the pointed Brill-Noether Theorem.

When $g \in \{0,1\}$, one has $\rho(g,r,d,\bm{a}^\ell(P)) \geq 0$ for any $(C,P)$ and any $\ell \in G^r_d(C)$. However, when $g \geq 2$ and   
 $\bm{a}: 0 \leq a_0 < \dots < a_r \leq d$  such that $\rho(g,r,d,\bm{a}) <0$, the locus 
 \[
 \M^{\bm{a}}_{g,d}:=\left\{ (C,P)\in \M_{g,1} \,|\, \bm{a}^\ell(P) \geq \bm{a} \mbox{ for some } \ell\in G^r_d(C)\right\}
 \] 
 is a proper subvariety of $\M_{g,1}$: when $\rho(g,r,d,\bm{a})=-1$, it contains a unique divisorial component, while all components have higher codimension if $\rho(g,r,d,\bm{a})<-1$ \cite{EH-1}. 
For example, when $d=2g-2$, \mbox{$r=g-1$,} and $\bm{a}=(0,1,2,\dots,g-2,g)$, the locus $\M^{\bm{a}}_{g,d}$ is the irreducible divisor of curves with a marked Weierstrass point.
In general, $\M^{\bm{a}}_{g,d}$ may not be irreducible, see \cite[\S 2]{EH-1}.

Assume $\rho(g,r,d,\bm{a}) =-1$, and denote by $\left[\ov{\M}^{\bm{a}}_{g,d}\right]\in \mathrm{Pic}\left( \ov{\M}_{g,1}\right)$ the class of the closure of the divisorial component of $\M^{\bm{a}}_{g,d}$.
After \cite{EH-1}, one has
\begin{equation}
\label{eq:muBNnuW}
\left[\ov{\M}^{\bm{a}}_{g,d}\right] = \mu_{g,d,\bm{a}} \, \mathcal{BN}_g + \nu_{g,d,\bm{a}} \, \mathcal{W}_g  \,\,\in\,\, \mathrm{Pic}\left( \ov{\M}_{g,1}\right),
\end{equation}
for some $\mu_{g,d,\bm{a}},\nu_{g,d,\bm{a}} \in \mathbb{Q}_{\geq 0}$, where $\mathcal{W}_g$ is the class of the divisor of curves with a marked Weierstrass point \cite{Cuk}:
\begin{equation}
\label{eq:W}
 \mathcal{W}_g  := {g+1 \choose 2}\psi - \lambda - \sum_{i=1}^{g-1} {g-i+1 \choose 2}\delta_i \,\,\in\,\, \mathrm{Pic}\left( \ov{\M}_{g,1}\right),
\end{equation}
and $\mathcal{B}\mathcal{N}_g$ is the pullback of the Brill-Noether divisor class from $\ov{\M}_g$ \cite{MR910206}:
\begin{equation}
\label{eq:BN}
\mathcal{BN}_g := (g+3)\lambda - \frac{g+1}{6}\delta_0 - \sum_{i=1}^{g-1} i(g-i)\delta_i \,\,\in\,\, \mathrm{Pic}\left( \ov{\M}_{g,1}\right).
\end{equation} 

Explicit formulae for $\mu_{g,d,\bm{a}} $ and $\nu_{g,d,\bm{a}}$ were computed in \cite{FT}, and make an appearance in our computations. These values are expressed in terms of 
 the number 
 \begin{equation}
 \label{eq:ngdadef}
 n_{g,d,\bm{a}}:= \#\left\{ (P, \ell) \in C \times G^r_d(C) \,| \, \bm{a}^\ell(P) \geq  \bm{a}\right\}, 
\end{equation}
 where $C$ is a general curve of genus $g \geq 2$. Let $\delta^i_j$ be the Kronecker delta.  After \cite{FT}, one has
\begin{equation}
\label{eq:ngda}
\begin{split}
 n_{g,d,\bm{a}} = g! \sum_{0 \leq k_1 < k_2 \leq r}& \left( (a_{k_2} - a_{k_1})^2 -1\right) \\ 
 &\times \frac{\prod_{0 \leq i < j \leq r} \left(a_j - \delta^{k_1}_j - \delta^{k_2}_j - a_i + \delta^{k_1}_i + \delta^{k_2}_i \right)}{\prod_{i=0}^r \left( g-d+r+a_i-\delta^{k_1}_i-\delta^{k_2}_i \right)!}
\end{split}
\end{equation}
and consequently, one has
\begin{align}
\label{eq:munu}
\begin{split}
\mu_{g,d,\bm{a}} &= -\frac{n_{g,d,\bm{a}}}{2(g^2-1)} + \frac{1}{4 \binom{g-1}{2}}\sum_{i=0}^r n_{g-1,d,\bm{a}^i},\\
\nu_{g,d,\bm{a}} &= \frac{n_{g,d,\bm{a}}}{g(g^2-1)}
\end{split}
\end{align}
where $\bm{a}^i:=\left(a_0+1-\delta_0^i, \dots, a_r + 1 - \delta_r^i \right)$.


\section{The incidence divisors and pointed Brill-Noether divisors}
\label{sec:proofThmskWkBN}

Here we prove Theorem \ref{thm:kW} and Corollary \ref{cor:kBN}. 
The argument involves  the incidence divisor ${\mathbb{H}}^k_{g, 1}\subset \P\E^k_{g,1}$ from \eqref{eq:incdiv}.
The closure  $\ov{\mathbb{H}}^k_{g, 1}$ in $\P\E^k_{g,1}$ is described by the incidence variety compactification introduced in \cite{bcggm1, bcggm}.
The class of $\ov{\mathbb{H}}^k_{g, 1}$ is given by:

\begin{lemma}[{\cites[\S1.6]{sauvagetcohomology}[\S4]{korotkin2019tau}}]
\label{lem:Z1class}
One has $\ov{\mathbb{H}}^k_{g, 1} \equiv k \,\psi -\eta$ in $\mathrm{Pic}\left( \P\E^k_{g,1}\right)$ for $k\geq 1$ and $g\geq 2$.
\end{lemma}

\noindent Here $\psi$ is the first Chern class of the cotangent line bundle at the marked point.

To prove Theorem \ref{thm:kW} and Corollary \ref{cor:kBN}, 
consider the forgetful morphisms: 
\[
\begin{tikzcd}
 & \P\E^k_{g,1} \arrow{dr}{\pi} \arrow{dl}[swap]{\varphi} \\
\ov{\M}_{g,1} && \P\E^k_g.
\end{tikzcd}
\]
The key step in the proof is the study of
the intersection in $\P\E^k_{g,1}$ of $\ov{\mathbb{H}}^k_{g, 1}$ and the pullbacks via $\varphi$ of pointed Brill-Noether divisors in $\ov{\M}_{g,1}$. 
Specifically, we have:

\begin{proposition}
\label{prop:BNirreducible} 
For $k\geq 1$, $g\geq 2$ and $\bm{a}: 0 \leq a_0 < \dots < a_r \leq d$  such that $\rho(g,r,d,\bm{a}) =-1$, one has
\[
\pi_*\left(\left[\ov{\mathbb{H}}^k_{g, 1}\right] \cdot \varphi^*\left[\ov{\M}^{\bm{a}}_{g,d}\right]\right) = \left[\ov{\H}^{\bm{a}}_{g,d}\right] \,\, \in\,\,\mathrm{Pic}\left( \P\E^k_{g} \right).
\] 
\end{proposition}

Before proving Proposition \ref{prop:BNirreducible}, we show how this implies Theorem \ref{thm:kW} and Corollary \ref{cor:kBN}. 
Recall the Weierstrass divisor class $\mathcal{W}_g$ from \eqref{eq:W}.

\begin{proof}[Proof of Theorem \ref{thm:kW}] 
For  $d=2g-2$, $r=g-1$, and \mbox{$\bm{a}=(0,1,\dots,g-2,g)$,} one has the specializations
$\ov{\H}^{\bm{a}}_{g,d} \equiv \ov{\W}^k_g $ in $\mathrm{Pic}\left(\P\E^k_{g} \right)$ and
$\ov{\M}^{\bm{a}}_{g,d} \equiv \mathcal{W}_g$ in $\mathrm{Pic}\left( \ov{\M}_{g,1} \right)$. Thus Proposition \ref{prop:BNirreducible} implies $\ov{\W}^k_g \equiv\pi_* \left( \left[\ov{\mathbb{H}}^k_{g, 1}\right]\cdot \varphi^*\,\mathcal{W}_g \right)$ in $\mathrm{Pic}\left( \P\E^k_{g} \right)$.
To compute this, consider first the intersection
\begin{multline*}
\left[\ov{\mathbb{H}}^k_{g, 1}\right]\cdot \varphi^*\,\mathcal{W}_g \\
= \left(k\,\psi -\eta\right) \left( {g+1 \choose 2} \psi - \lambda -\sum_{i=1}^{g-1}{g-i+1 \choose 2}\delta_i \right) \,\,\in\,\, A^2\left(\P\E^k_{g,1}\right),
\end{multline*}
where we used Lemma \ref{lem:Z1class}.
The push-forward via $\pi\colon \P\E^k_{g,1} \rightarrow\P\E^k_{g}$ is 
\begin{multline*}
\pi_* \left(\left[\ov{\mathbb{H}}^k_{g, 1}\right]\cdot \varphi^*\,\mathcal{W}_g \right) =  -g(g^2-1)\,\eta + k {g+1 \choose 2} \kappa_1 -k(2g-2)\,\lambda \\
 -\sum_{i=1}^{\lfloor g/2 \rfloor} k\left( (2i-1) {g-i+1 \choose 2} +(2g-2i-1) {i+1 \choose 2} \right)\delta_i
\end{multline*}
in $\mathrm{Pic}\left( \P\E^k_{g} \right)$.
Here, we used 
\begin{align*}
\kappa_1:=\pi_* \left(\psi^2\right), \qquad \pi_* \left(\psi\lambda\right)=(2g-2)\lambda, \qquad \pi_* \left(\psi\eta\right)=(2g-2)\eta, \\
\pi_* \left(\psi\delta_i\right)=(2i-1)\delta_i \,\,\,\mbox{and}\,\,\, \pi_* \left(\psi\delta_{g-i}\right)=(2g-2i-1)\delta_i \,\,\,\mbox{for $1\leq i \leq \lfloor g/2 \rfloor$.}
\end{align*}
Mumford's formula $\kappa_1=12 \lambda -\sum_{i=0}^{\lfloor g/2 \rfloor} \delta_i$ and simplifying yield
\begin{align*}
\pi_* \left(\left[ \ov{\mathbb{H}}^k_{g, 1}\right]\cdot \varphi^*\,\mathcal{W}_g \right) =& -g(g^2-1)\,\eta + k(6g^2+4g+2)\,\lambda - k{g+1 \choose 2}\delta_0 \\
&-\sum_{i=1}^{\lfloor g/2 \rfloor} k(g+3)i(g-i)\,\delta_i.
\end{align*}
The statement follows.
\end{proof}

\begin{proof}[Proof of Corollary \ref{cor:kBN}]
From Proposition \ref{prop:BNirreducible}, Lemma \ref{lem:Z1class}, and \eqref{eq:muBNnuW}, one has
\begin{align*}
\left[\ov{\H}^{\bm{a}}_{g,d}\right]
&=\pi_*\left( \left[\ov{\mathbb{H}}^k_{g, 1} \right]\cdot \varphi^*\left[ \ov{\M}^{\bm{a}}_{g,d} \right]\right) \\
&= \pi_*\big(\left(k\,\psi-\eta\right)\left(\mu_{d,g,\bm{a}}\, \mathcal{BN}_g+\nu_{d,g,\bm{a}}\,\mathcal{W}_g \right) \big)\\
&= k(2g-2)\,\mu_{d,g,\bm{a}}\, \mathcal{BN}_g+\nu_{d,g,\bm{a}}\,\pi_* \left(\left[\ov{\mathbb{H}}^k_{g, 1}\right]\cdot \varphi^*\,\mathcal{W}_g \right)
\end{align*}
in $\mathrm{Pic}\left( \P\E^k_{g} \right)$.
Again by Proposition \ref{prop:BNirreducible}, one has $\pi_* \left( \left[\ov{\mathbb{H}}^k_{g, 1}\right]\cdot \varphi^*\,\mathcal{W}_g \right)\equiv \ov{\W}^k_g$ in $\mathrm{Pic}\left( \P\E^k_{g} \right)$, hence the statement.
\end{proof}

Expanding the formula in Corollary \ref{cor:kBN}, we deduce:

\begin{corollary}
\label{cor:expcoeffBNdiv}
For $k\geq 1$, $g \geq 2$,  and   
 $\bm{a}: 0 \leq a_0 < \dots < a_r \leq d$  such that $\rho(g,r,d,\bm{a}) =-1$,
one has
\[
\left[\ov{\H}^{\bm{a}}_{g,d}\right] = c_\eta\,\eta + c_\lambda\,\lambda - \sum_{i=0}^{\lfloor g/2 \rfloor}c_i\,\delta_i \,\in\,\mathrm{Pic}\left(\P\E^k_{g} \right),
\]
where
\begin{align*}
c_\eta &= -g(g^2-1)\,\nu_{g,d,\bm{a}}, \\
c_\lambda &= 2(g-1)(g+3)k\,\mu_{g,d,\bm{a}} + 2(3g^2 + 2g +1)k\,\nu_{g,d,\bm{a}}, \\ 
c_0 &=  \frac{g^2-1}{3}k\,\mu_{g,d,\bm{a}} + \frac{g(g+1)}{2}k\,\nu_{g,d,\bm{a}}, \\ 
c_i &= 2i(g-i)(g-1)k\,\mu_{g,d,\bm{a}} + i(g-i)(g+3)k\,\nu_{g,d,\bm{a}} \quad \mbox{for $i\geq 1$,}
\end{align*}
and $\mu_{g,d,\bm{a}}$ and $\nu_{g,d,\bm{a}}$ are given by \eqref{eq:munu}.
\end{corollary}

We now turn to the proof of Proposition \ref{prop:BNirreducible}. 

\begin{proof}[Proof of Proposition \ref{prop:BNirreducible}] 
By definition, the locus ${\H}^{\bm{a}}_{g,d}$ is the only component of $\pi\left(\ov{\mathbb{H}}^k_{g, 1} \cap \varphi^{-1}\left(\ov{\M}^{\bm{a}}_{g,d}\right)\right)$ in the restriction of $\P\E^k_{g}$ over the locus of \textit{smooth} curves $\M_g$. Also, $\pi\colon {\mathbb{H}}^k_{g, 1} \cap {\M}^{\bm{a}}_{g,d}\rightarrow {\H}^{\bm{a}}_{g,d}$ is generically of degree one. This implies the statement over $\M_g$.

Using the exact sequence
\[
\mathrm{Pic}\left(\P\E^k_{g}\big|_{\ov{\M}_g\setminus \M_g} \right) \rightarrow \mathrm{Pic}\left(\P\E^k_{g} \right) \rightarrow \mathrm{Pic}\left(\P\E^k_{g}\big|_{\M_g} \right) \rightarrow 0,
\]
we deduce that one has
\begin{equation}
\label{eq:E}
\pi_*\left(\left[\ov{\mathbb{H}}^k_{g, 1}\right] \cdot \varphi^* \left[\ov{\M}^{\bm{a}}_{g,d}\right]\right) = \left[\ov{\H}^{\bm{a}}_{g,d}\right] +E \qquad \in\mathrm{Pic}\left( \P\E^k_{g} \right)
\end{equation}
for some  cycle $E$ which is either zero, or supported only over the divisor $\ov{\M}_g\setminus \M_g$  of singular curves. 
Write 
\[
\P\E^k_{g}\big|_{\ov{\M}_g\setminus \M_g}=\Delta_0\cup \cdots\cup\Delta_{\lfloor g/2 \rfloor},
\]
where $\Delta_0$ is the divisor whose general element consists of an irreducible nodal curve together with the class of a general stable $k$-differential, and $\Delta_i$ is the divisor 
whose general element consists of a reducible nodal curve with two components of genera $i$ and $g-i$ together with the class of a general stable $k$-differential,  for $1 \leq i \leq \lfloor g/2 \rfloor$.
To show that $E=0$, 
it is enough to argue that the preimage under $\pi$ of a \textit{general} element of $\Delta_i$, for $0 \leq i \leq \lfloor g/2 \rfloor$, is disjoint from $\ov{\mathbb{H}}^k_{g, 1} \cap \varphi^{-1}\left(\ov{\M}^{\bm{a}}_{g,d}\right)$. 

First, consider the case $i > 0$. Let $(X, \mu)$ be a general element of $\Delta_i$, i.e., $X$ is a nodal curve  obtained by identifying the marked points of two \textit{general} pointed curves $(C_1, Q_1)$ and $(C_2, Q_2)$ of  genus $i$ and  $g-i$, respectively, and $\mu$ is the class of a \textit{general} stable $k$-differential on $X$.
There are finitely many points on $X$ which are limits of a Brill-Noether special point $P$ on a nearby smooth curve $C$ such that $\bm{a}^{\ell}(P) \geq \bm{a}$ for some $\ell\in G^r_d(C)$. Such limit points are away from the node of $X$. That is, the stable pointed curve  obtained when a marked point in $X$ collides with the node, thus creating a rational bridge containing the marked point, is not a limit of Brill-Noether special pointed curves \cite[Theorem (1.1)]{MR910206}.
In particular, the codimension-two locus in $\pi^{-1}\left( \Delta_i\right)\subset \P\E^k_{g,1}$ whose general element consists of a curve with a pointed rational bridge and general stable $k$-differential is not a component of the intersection $\ov{\mathbb{H}}^k_{g, 1} \cap \varphi^{-1}\left( \ov{\M}^{\bm{a}}_{g,d}\right)$.
Moreover, since $\mu$ is general, the zeros of $\mu$  avoid the limits in $X$ of Brill-Noether special points. 
Indeed, $\mu$ is a general element of 
\begin{equation*}
\P \left(H^0\left(\omega^{\otimes k}_{C_1}\left(kQ_1\right)\right) \times H^0\left(\omega^{\otimes k}_{C_2}\left(kQ_2\right)\right)\right)
\end{equation*}
satisfying the \textit{$k$-residue condition}.  
That is, $\mu$ is the class of a pair of general $k$-differentials $\mu_1$ and $\mu_2$ on the two components $C_1$ and $C_2$ with poles of order $k$ at $Q_1$ and $Q_2$, respectively, and with appropriate scaling  such that
\[
\mathrm{Res}^k_{Q_1}(\mu_1) = (-1)^k \,\mathrm{Res}^k_{Q_2}(\mu_2).
\]
Here, since $\mu_i$ has a pole of order $k$ at $Q_i$, the $k$-residue $\mathrm{Res}^k_{Q_i}(\mu_i)$ is computed as the coefficient of $t_i^{-k}(dt_i)^k$ in the Laurent series expansion of $\mu_i$ at $Q_i$, for a formal coordinate $t_i$ at $Q_i$, with $i=1,2$. 
Hence, the zeros of $\mu$, equal to the union of the zeros of $\mu_1$ and $\mu_2$, are generically away from the limits in $X$ of Brill-Noether special points on nearby smooth curves.

Finally, consider the case $i=0$.
Let $(X,\mu)$ be a general element of $\Delta_0$ in $\P\E^k_{g}$, i.e., $X$ is an irreducible nodal curve and $\mu$ is the class of a general stable $k$-differential on $X$.
Pulling back via the normalization map $\widetilde{X}\rightarrow X$, a general stable $k$-differential on $X$ is identified with a general $k$-differential on $\widetilde{X}$ having poles of
order $k$ at the two preimages of the node.
First,  consider the stable pointed curve $(X_0, P_0)$ obtained when a marked point in $X$ collides with the node, thus creating a rational component containing the marked point $P_0$ and meeting the rest of the curve in two points.
We claim that the codimension-two locus in $\pi^{-1}\left( \Delta_0\right)\subset \P\E^k_{g,1}$ whose general element consists of such a pointed curve $(X_0, P_0)$ together with 
a general stable $k$-differential $\mu_0$ is not a component of the intersection $\ov{\mathbb{H}}^k_{g, 1} \cap \varphi^{-1}\left( \ov{\M}^{\bm{a}}_{g,d}\right)$.
Indeed, the restriction of $\mu_0$ to each component of $X_0$  has
poles of order $k$ at the two nodes. It follows that the restriction of $\mu_0$ to the rational component has no zeros for degree reasons, hence $\mu_0$ does not vanish at ~$P_0$. This implies that $(X_0, P_0, \mu_0)$ does not lie in $\ov{\mathbb{H}}^k_{g, 1}$, hence the claim.
It remains to consider marked points in $X$ that are away from the node.
For this, one argues similarly to the case $i>0$, as the zeros of $\mu$ generically avoid  the finitely many limits in $X$ of Brill-Noether special points on nearby smooth curves.

It follows that $E=0$ in \eqref{eq:E}, hence the statement.
\end{proof}


\section{Test curves} 
\label{sec:testfamilies}
We consider here the intersections of the loci $\ov{\H}^{\bm{a}}_{g,d}$ with various  curves in $\P\E^k_{g}$. These will serve a double purpose:  We provide  explicit examples of  restrictions of the loci $\ov{\H}^{\bm{a}}_{g,d}$ to  families of curves with $k$-differentials, and the resulting computations provide a second independent proof of Theorem~\ref{thm:kW} and Corollary~\ref{cor:expcoeffBNdiv}.

The classes $\eta, \lambda$, and $\delta_i$ for $i=0,\dots, \lfloor g/2 \rfloor$ freely generate $\mathrm{Pic}\left(\P\E^k_g \right)$ for $g\geq 3$, while for $g=2$,  $\mathrm{Pic}\left(\P\E^k_2 \right)$ is generated by $\eta, \lambda,\delta_0,\delta_1$ modulo the relation $10\lambda = \delta_0 + 2\delta_1$.
Hence, given  $g\geq 2$ and  $\bm{a}: 0 \leq a_0 < \dots < a_r \leq d$ such that $\rho(g, r, d, \bm{a}) =-1$,
we can write
\begin{equation}
\label{eq:Hcoeff}
\left[ \ov{\H}^{\bm{a}}_{g,d} \right] = c_\eta \,\eta + c_\lambda \,\lambda - \sum_{i=0}^{\lfloor g/2 \rfloor} c_i \,\delta_i \,\,\in\,\, \mathrm{Pic}_{\mathbb{Q}}\left(\P\E^k_{g}\right),\quad \mbox{for some $c_\eta, c_\lambda, c_i\in \mathbb{Q}$}.
\end{equation}
All the coefficients $c_\eta, c_\lambda, c_i$, with $0\leq i\leq \lfloor g/2 \rfloor$, are given by Theorem \ref{thm:kW} and Corollary \ref{cor:expcoeffBNdiv}. The intersections with the following  families allow one to independently recover these coefficients.

\subsection{A pencil of $k$-canonical divisors on a general curve} 
\label{sec:ceta}
Let $C$ be a general genus $g$ curve $k$-canonically embedded in $\P^{N}$, where $N=g-1$ if $k=1$, or $N=(g-1)(2k-1) -1$ if $k\geq 2$. Let $\Lambda \cong \P^{N-3}$ be a fixed general subspace in $\P^{N}$, and consider the one-dimensional family  of hyperplanes in $\P^N$ containing $\Lambda$. 
As every such hyperplane cuts a $k$-canonical divisor on $C$,
this family gives rise to a pencil of $k$-canonical divisors on $C$, hence a pencil $A$ in $\P\E^k_g$.
One has 
\[
A \cdot \eta = -1, \quad A \cdot \lambda = 0,  \quad A \cdot \delta_i = 0, \quad \mbox{for all $i$.} 
\]
The pencil $A$ intersects $\ov{\H}^{\bm{a}}_{g,d}$ transversely along the Brill-Noether special points $P$ in $C$ such that $\bm{a}^\ell(P) \geq  \bm{a}$ for some $\ell\in G^r_d(C)$, hence one has
\[
A \cdot \left[\ov{\H}^{\bm{a}}_{g,d}\right] = n_{g,  d, \bm{a}} 
\]
where $n_{g,d,\bm{a}}$ is the number defined in \eqref{eq:ngdadef}. Combining with \eqref{eq:munu}, this gives 
\begin{equation}
\label{eq:ceta}
c_\eta =  -g(g^2-1)\,\nu_{g,d,\bm{a}}.
\end{equation}

\subsection{Curves on K3 surfaces}
\label{sec:curvesonK3}
Next, we consider  a Lefschetz pencil of curves of genus $g\geq 3$ lying on a general K3 surface $S$ of degree $2g-2$ in $\mathbb{P}^g$. 
Let \mbox{$\mathcal{X}\rightarrow S$} be the blow-up at the $2g-2$ base-points of the pencil, and let \mbox{$\pi\colon\mathcal{X}\rightarrow \mathbb{P}^1$} be the corresponding family of curves.
Fix $k$ general sections in $S$ and let $\Sigma_1, \dots, \Sigma_k$ be their proper transform in $\mathcal{X}$. Since $S$ has canonical sections, the cycle $\Sigma_1+\cdots+\Sigma_k$ on $\mathcal{X}$ restricts to a $k$-canonical divisor on each fiber of $\pi$, hence this gives rise to a pencil $\tau\colon\mathbb{P}^1\rightarrow \P\E^k_{g}$.
The intersections with the generators are 
\[
\tau^* \eta = k, \qquad
\tau^* \lambda = g+1, \qquad
\tau^* \delta_0 = 6g+18, \qquad
\tau^* \delta_i = 0, \qquad \mbox{for $i>0$.}
\]
The intersections with $\lambda$ and the $\delta_i$ are classical \cite{cukierman1993curves, MR2123229}. The degree of $\eta$ can be computed by intersecting the relation $\omega_{\mathcal{X}/\P^1}^{\otimes k} = \pi^*\eta \otimes \mathcal{O}_\mathcal{X}(\Sigma_1 + \cdots + \Sigma_k)$ valid on $\mathcal{X}$ 
with the class of one of the $2g-2$ exceptional divisors (this is as in \cite{chen2018positivity}, see also \cite[Example 3.2]{gheorghita2018effective} for  a similar computation). 

To compute the intersection of the pencil $\tau$ with the divisor class $\left[\ov{\H}^{\bm{a}}_{g,d}\right]$, we start from the locus of corresponding Brill-Noether special points  in $\mathcal{X}$. 
After identifying $\ov{\M}_{g,1} \rightarrow \ov{\M}_{g}$ with the universal curve over $\ov{\M}_{g}$,
one has a moduli map $\mathcal{X}\rightarrow \ov{\M}_{g,1}$. The pull-back of the divisor class $\left[\ov{\M}_{g,d}^{\bm{a}}\right]$ on $\ov{\M}_{g,1}$ from \eqref{eq:muBNnuW} via such map is
\begin{align}
\label{eq:Xagd}
\begin{split}
\left[\calX^{\bm{a}}_{g,d}\right]:= &\nu_{g,d,\bm{a}}\,{g+1 \choose 2}\, c_1\left(\omega_{\mathcal{X}/\mathbb{P}^1}\right) \\
& +(\mu_{g,d,\bm{a}}\,(g+3) -\nu_{g,d,\bm{a}})\,  \lambda - \mu_{g,d,\bm{a}}\, \frac{g+1}{6}\,  \delta_0.
\end{split}
\end{align}
It follows that
\[
\tau^* \left[\ov{\H}^{\bm{a}}_{g,d}\right]  = \left[\calX^{\bm{a}}_{g,d}\right]\cdot \left([\Sigma_1]+ \cdots+[\Sigma_k]\right).
\]
By the adjunction formula, one has $\Sigma_j \equiv K_f \equiv f+ K_{\mathcal{X}}\equiv f+\sum_{i=1}^{2g-2}E_i$, for each $j=1,\dots,k$, 
where $f$ is  a fiber of $\pi$, and $E_i$ are the exceptional curves on~$\mathcal{X}$. 
This gives 
\begin{align}
\label{eq:K3Hagd}
\tau^* \left[\ov{\H}^{\bm{a}}_{g,d}\right]  = \left[\calX^{\bm{a}}_{g,d}\right]\cdot k\left[ f+\sum_{i=1}^{2g-2}E_i\right] =2k(g+1)(g-1)^2\,\nu_{g,d,\bm{a}}.
\end{align}
Here we used that the nonzero intersections are given by 
\begin{align*}
c_1\left(\omega_{\mathcal{X}/\mathbb{P}^1}\right) \cdot f = 2g-2, \quad
c_1\left(\omega_{\mathcal{X}/\mathbb{P}^1}\right) \cdot E_i=1,  \quad
\lambda\cdot E_i = \tau^* \lambda= g+1, \\
\delta_0\cdot E_i = \tau^* \delta_0= 6g+18, \mbox{ for each $i=1,\dots,2g-2$}.
\end{align*}
Note how \eqref{eq:K3Hagd} is independent of $\mu_{g,d,\bm{a}}$, as the pencil has zero intersection with the Brill-Noether class $\mathcal{BN}_g$ \cite{MR852158}.
It follows that the intersection with \eqref{eq:Hcoeff} gives
\begin{equation}
\label{eq:K3}
2k(g+1)(g-1)^2\,\nu_{g,d,\bm{a}} = k\,c_\eta +(g+1) \, c_\lambda - (6g+18)\,c_0.
\end{equation}

\subsection{A Du Val pencil}
We consider here a pencil of Du Val curves of genus $g\geq 2$, as introduced in \cite{arbarello2016explicit}.
We briefly review the definition. Let $S'$ be the blow-up of $\P^2$ at nine general points $P_1,\dots, P_9$ (the generality condition is explicitly given in \cite{arbarello2016explicit}), and consider the linear system on $S'$
\[
L_g:=\left|3g \ell -g\left( E_ 1 + \cdots + E_ 8\right) -(g-1) E_9 \right|
\]
where $\ell$ is the proper transform of a line in $\P^2$, and $E_1,\dots, E_9$ are the exceptional curves in $S'$.
There exists a unique curve $J'\equiv 3\ell -(E_ 1 + \cdots + E_ 9)$ in $S'$ corresponding to the unique smooth plane cubic curve passing through the points $P_1,\dots,P_9$, 
and for each $C'\in L_g$, one has $C'\cdot J'=1$. It follows that $C'\cap J'=:\{P_{10}\}$ is independent of $C'$, and thus $P_{10}$ is a base-point of~$L_g$.

Let $S\rightarrow S'$ be the blow-up at $P_{10}$. Denote still by $E_1,\dots, E_9$ the inverse image in $S$ of the exceptional curves in $S'$, and let $E_{10}$ be the exceptional curve in $S$ over $P_{10}$. Let $J$ and $C$ be the proper transforms of $J'$ and $C'\in L_g$. One has
\begin{align*}
-K_S &\equiv J \equiv 3\ell -(E_ 1 + \cdots + E_ {10}),\\
C &\equiv 3g \ell -g\left( E_ 1 + \cdots + E_ 8\right) -(g-1) E_9 - E_{10}.
\end{align*}
The linear system $|C|$ is base-point free and gives rise to a map \mbox{$S\rightarrow \overline{S}\subset \P^g$} contracting $J$ into an elliptic singularity. The surface $\overline{S}$ has canonical sections and deforms in $\P^g$ to a nonsingular K3 surface of degree $2g-2$ \cite{arbarello2015hyperplane}.

Consider a Lefschetz pencil in $|C|$. Let $\mathcal{X}\rightarrow S$ be the blow-up at the $C^2=2g-2$ base-points of the pencil, and let $\pi\colon \mathcal{X}\rightarrow \P^1$ be the corresponding family of curves. 
As in \S\ref{sec:curvesonK3}, consider the proper transforms $\Sigma_1, \dots, \Sigma_k$ in $\mathcal{X}$ of $k$ general sections in $S$.  Since $\ov{S}$ has canonical sections, the cycle $\Sigma_1+\cdots+\Sigma_k$ cuts $k$-canonical divisors on the fibers of $\pi$, hence this gives rise to a pencil $\tau\colon\mathbb{P}^1\rightarrow \P\E^k_{g}$.
The intersections with the generators are 
\begin{align*}
\tau^* \eta &= k, &
\tau^* \lambda &= g, &
\tau^* \delta_0 &= 6(g+1), \\
&&\tau^* \delta_1 &= 1, & 
\tau^* \delta_i &= 0,  \qquad\mbox{for $i>1$.}
\end{align*}
The intersections with $\lambda$ and the $\delta_i$ are computed in \cite{arbarello2016explicit}, and the intersection with $\eta$ is computed as in \S\ref{sec:curvesonK3}.
We note that there is a unique fiber $J+D$ of $\pi$ contributing to $\tau^* \delta_1$ and consisting of  the elliptic curve $J$   and a $g-1$ curve 
\[
D\equiv 3(g-1) \ell -(g-1)\left( E_ 1 + \cdots + E_ 8\right) -(g-2) E_9
\]
meeting in one point.

Consider the moduli map $\gamma\colon\mathcal{X}\rightarrow \ov{\M}_{g,1}$. Since
$\gamma^*\delta_1 =J$ and $\gamma^*\delta_{g-1} =D$,
the pull-back via $\gamma$ of the divisor class $\left[\ov{\M}_{g,d}^{\bm{a}}\right]$ expressed by \eqref{eq:muBNnuW}  is
\begin{align*}
\left[\calX^{\bm{a}}_{g,d}\right]:= &\nu_{g,d,\bm{a}}\,{g+1\choose 2}\, c_1\left(\omega_{\mathcal{X}/\mathbb{P}^1}\right) 
 +(\mu_{g,d,\bm{a}}\,(g+3) -\nu_{g,d,\bm{a}})\, \lambda \\
 &- \mu_{g,d,\bm{a}}\, \frac{g+1}{6}\, \delta_0
-\left(\mu_{g,d,\bm{a}}\,(g-1)+ \nu_{g,d,\bm{a}}\,{g \choose 2}\right)J\\
& -\left(\mu_{g,d,\bm{a}}\,(g-1) + \nu_{g,d,\bm{a}}\right)D.
\end{align*}
The intersection with the divisor class $\left[\ov{\H}^{\bm{a}}_{g,d}\right]$ is 
\[
\tau^* \left[\ov{\H}^{\bm{a}}_{g,d}\right]  = \left[\calX^{\bm{a}}_{g,d}\right]\cdot \left([\Sigma_1]+ \cdots+[\Sigma_k]\right).
\]
By the adjunction formula, one has 
\[
\Sigma_j \equiv K_f \equiv  f+ K_{\mathcal{X}}\equiv  f+K_S+\sum_{i=1}^{2g-2}F_i \equiv f-J+\sum_{i=1}^{2g-2}F_i,
\]
for each $j=1,\dots,k$, 
where $f$ is  a fiber of $\pi$, and $F_i$ are the exceptional curves over the base-points of the pencil in~$S$. 
This gives 
\begin{align}
\label{eq:DVHagd}
\tau^* \left[\ov{\H}^{\bm{a}}_{g,d}\right]  = \left[\calX^{\bm{a}}_{g,d}\right]\cdot k\left[ f-J+\sum_{i=1}^{2g-2}F_i\right] = k(2g-3)(g^2-1) \,\nu_{g,d,\bm{a}}.
\end{align}
The result is independent of $\mu_{g,d,\bm{a}}$, and indeed the pencil has zero intersection with the class $\mathcal{BN}_g$ \cite{arbarello2016explicit}.
In \eqref{eq:DVHagd}, we used that the nonzero intersections are given by 
\begin{align*}
c_1\left(\omega_{\mathcal{X}/\mathbb{P}^1}\right) \cdot f = 2g-2, \quad
c_1\left(\omega_{\mathcal{X}/\mathbb{P}^1}\right) \cdot F_i=1,  \quad
\lambda\cdot F_i = \tau^* \lambda= g, \\
\delta_0\cdot F_i = \tau^* \delta_0= 6(g+1),\quad
D\cdot F_i = \tau^* \delta_{1}= 1, \quad
 \mbox{for  $i=1,\dots,2g-2$}.
\end{align*}
For the intersection $D\cdot F_i$, note that $C\cdot D=2g-2$, hence the restrictions of all base-points of the pencil to the fiber $J+D$ of $\pi$ are in $D$.

Intersecting with \eqref{eq:Hcoeff}, we deduce the following relation on the coefficients of the  class $\left[\ov{\H}^{\bm{a}}_{g,d}\right]$:
\begin{equation}
\label{eq:DV}
k(2g-3)(g^2-1) \,\nu_{g,d,\bm{a}} = k\,c_\eta +g \, c_\lambda - 6(g+1)\,c_0 - \, c_1.
\end{equation}

\subsection{A pencil of hyperelliptic curves}
\label{sec:hyppencil}
We consider here a pencil of hyperelliptic curves, following \cite[pg.~361-363]{MR1070600}. Let $S\rightarrow \P^2$ be a double cover branched over a general smooth curve of degree $2g+2$ in $\P^2$. This gives a two-dimensional family of hyperelliptic curves of genus $g$. Consider a general pencil of hyperelliptic curves in $S$. 
 Such a pencil has two base points  \cite[pg.~361-363]{MR1070600}.
Let $\pi\colon\mathcal{X} \ra \P^1$ be the corresponding family of curves obtained by blowing-up the two base points.  
A choice of $k(g-1)$ lines in $\P^2$ gives a $k$-canonical divisor on each curve in the family after pulling back to $\mathcal{X}$. This gives rise to a pencil in $\P\E^k_g$. One has
\begin{align*}
\deg\eta &= k,  & \deg\lambda &=\frac{g(g+1)}{2}, \\
\deg\delta_0& = 2(g+1)(2g+1), & \deg\delta_i&=0 &\mbox{for $i\geq 1$.}
\end{align*}
The intersections with $\lambda$ and $\delta_i$ are computed  in \cite[pg.~361-363]{MR1070600}, and the intersection with $\eta$ can be computed as in \S \ref{sec:curvesonK3}.

Furthermore, let $\Sigma_i$, for $i=1,\dots, k(g-1)$, be the pullbacks to $\mathcal{X}$ of the fixed $k(g-1)$ lines in $\P^2$, let $f$ be a fiber of $\pi$, and  $E_1,E_2$ be the  two exceptional curves. By the adjunction formula, one has $\Sigma_i \equiv K_f \equiv f+E_1 + E_2$ on $\mathcal{X}$.
The pointed Brill-Noether divisor class $\big[\calX^{\bm{a}}_{g,d}\big]$ on $\mathcal{X}$ is as in  \eqref{eq:Xagd}.
Similarly to \S \ref{sec:curvesonK3}, the intersection of the pencil $\pi\colon\mathcal{X} \ra \P^1$ with the divisor class $\left[\ov{\H}^{\bm{a}}_{g,d}\right]$ equals 
\begin{align*} 
\big[\calX^{\bm{a}}_{g,d}\big] \cdot & \sum_{i=1}^{k(g-1)} [\Sigma_i ]= \big[\calX^{\bm{a}}_{g,d}\big] \cdot k(g-1)(f+E_1 + E_2) \\
&= kg(g+1)(g-1)^2 \,\nu_{g,d,\bm{a}} - \frac{1}{3}k(g+1)(g-1)^2(g-2)\,\mu_{g,d,\bm{a}}.
\end{align*}
Here we used that the nonzero intersections are given by 
\begin{align*}
c_1\left(\omega_{\mathcal{X}/\mathbb{P}^1}\right) \cdot f = 2g-2,\quad  
c_1\left(\omega_{\mathcal{X}/\mathbb{P}^1}\right) \cdot E_i=1, \quad 
\lambda\cdot E_i =\frac{g(g+1)}{2}, \\
\delta_0\cdot E_i = 2(g+1)(2g+1), \mbox{ for each $i=1,2$.}
\end{align*}
 It follows that the coefficients of the  class $\left[\ov{\H}^{\bm{a}}_{g,d}\right]$ in \eqref{eq:Hcoeff} satisfy
\begin{equation}
\begin{split}
\label{eq:hyp}
kg(g+1)(g-1)^2 \,\nu_{g,d,\bm{a}} -& \frac{1}{3}k(g+1)(g-1)^2(g-2)\,\mu_{g,d,\bm{a}} \\
&= k\,c_\eta +\frac{g(g+1)}{2}\,c_\lambda - 2(g+1)(2g+1)\,c_0.
\end{split}
\end{equation}

\subsection{Pull-back to $\overline{\mathcal{M}}_{0,g}$}
For $g\geq 4$, we define here a map
\[
\xi\colon \overline{\mathcal{M}}_{0,g} \rightarrow \P\E^k_g.
\]
Consider an elliptic curve $J$ with distinguished point $Q\in J$, and let $\mu_J$ be a general $k$-differential on $J$ with a pole of order $k-1$ at $Q$ and regular elsewhere.
Specifically, given a degree $d>0$, it will be enough to select $\mu_J\in H^0\left(J, \mathscr{O}_J((k-1)Q)\right)$ which does not vanish at the finitely many points $P\in J$ such that $\mathscr{O}_J(P-Q)$ is an $e$-torsion point in $\mathrm{Pic}^0(J)$ for some $e\leq d$.
Then $\xi$ is defined by mapping
 a stable $g$-pointed rational curve to  the stable curve of genus $g$ obtained by attaching $g$ copies of $J$ at the $g$ marked points, and assigning the stable $k$-differential which restricts as $\mu_J$ on each copy of $J$ and is zero on all rational components.

Let $\epsilon_i$ be the class of the divisor in $\overline{\mathcal{M}}_{0,g}$  whose generic element consists of two rational components, one of them containing exactly $i$ marked points, for $2\leq i \leq \lfloor g/2 \rfloor$. One has $\xi^* \eta =0$, and from \cite{MR910206}:
\[
 \xi^* \lambda =0, 
\quad
\xi^* \delta_1 =-\sum_{i=2}^{\lfloor g/2 \rfloor} \frac{i(g-i)}{g-1}\epsilon_i,
\quad
 \xi^* \delta_i=\epsilon_i, \,\, \mbox{ for $2\leq i \leq \lfloor g/2 \rfloor$}.
\]

\begin{lemma}
For $g\geq 4$ and  $\bm{a}: 0 \leq a_0 < \dots < a_r \leq d$ such that $\rho(g, r, d, \bm{a}) <0$,
the  locus $\ov{\H}^{\bm{a}}_{g,d}$ in $\P\E^k_g$ is disjoint from the image of $\xi$. 
In particular, when $\rho(g, r, d, \bm{a}) =-1$, it follows that
$\xi^*\left[\ov{\H}^{\bm{a}}_{g,d}\right]=0$ in $\mathrm{Pic}\left( \overline{\mathcal{M}}_{0,g} \right)$, hence one has
\begin{equation}
\label{eq:ci}
c_i=\frac{i(g-i)}{g-1}\,c_1, \qquad \mbox{for $2\leq i \leq \lfloor g/2 \rfloor$}.
\end{equation}
\end{lemma}

\begin{proof}
Let $(X,\mu)$ in $\P\E^k_g$ be an arbitrary element in the image of $\xi$.
For $(X,\mu)$ to be in $\ov{\H}^{\bm{a}}_{g,d}$, there need to be a twisted $k$-differential of type $\mu$ (in the sense of \cite{bcggm1, bcggm}) vanishing at a limit in $X$ of a Brill-Noether special point of some linear series of type $\mathfrak{g}^r_d$ on a nearby smooth curve.
Since the restriction $\mu_J$ of $\mu$ to each elliptic tail of $X$ is generically nonzero, all twisted $k$-differentials of type $\mu$  
restrict as $\mu_J$ on each elliptic tail.
 Points on rational components of  $X$  are always Brill-Noether general \cite{EH-1}, and Brill-Noether special points of linear series of type $\mathfrak{g}^r_d$
 specialize to points $P$ on elliptic tails of $X$ such that $\mathscr{O}_J(P-Q)$ is an $e$-torsion point in $\mathrm{Pic}^0(J)$ for some $e\leq d$ \cites{MR846932}[Lemma 2.1]{osserman2014simple}.
 By assumption, $\mu_J$ vanishes away from such points, hence
 the image of the map $\xi$ is disjoint from the locus $\ov{\H}^{\bm{a}}_{g,d}$ in $\P\E^k_g$.
\end{proof}

\subsection{Second proof of Theorem \ref{thm:kW} and Corollary \ref{cor:expcoeffBNdiv}}
Expanding the class $\left[\ov{\H}^{\bm{a}}_{g,d}\right]$ as in \eqref{eq:Hcoeff}, the relations in \eqref{eq:ceta}, \eqref{eq:K3}, \eqref{eq:DV}, \eqref{eq:hyp}, \eqref{eq:ci}
allow one to determine all the coefficients $c_\eta, c_\lambda, c_i$, with $0\leq i\leq \lfloor g/2 \rfloor$ modulo the required relation $10\lambda = \delta_0 + 2\delta_1$ in the case $g=2$.
\hfill$\square$


\section{Extremality and rigidity of the incidence divisor}

In this section we prove Theorem \ref{thm:Z1extremal}. To deduce extremality, we  use: 

\begin{lemma}[{\cite[Lemma 4.1]{MR3071469}}]
\label{lem:extremality}
Let $D$ be an irreducible effective divisor in a projective variety~$X$, and let $\mathscr{C}$ be a set of irreducible effective curves contained in $D$ such that $\bigcup_{C\in \mathscr{C}} C$ is Zariski dense in $D$. If for every curve $C$ in $\mathscr{C}$ one has
\[ 
C \cdot (D + B) \leq 0, \qquad \mbox{for a fixed big divisor class $B$ on $X$,}
\] 
then $D$ is extremal in $\ov{\mathrm{Eff}}^1(X)$.
\end{lemma}

To review the notation, a divisor class on $X$ is \textit{big} if it lies in the interior of the pseudo-effective cone $\ov{\mathrm{Eff}}^1(X)$.
We emphasize that in Lemma \ref{lem:extremality}, the curves in $\mathscr{C}$ are not required to be moving in $D$. We apply  Lemma \ref{lem:extremality} below to deduce the extremality of the divisor $\ov{\mathbb{H}}^k_{g, 1}$ using a set $\mathscr{C}$ of Teichm\"uller curves  in $\ov{\mathbb{H}}^k_{g, 1}$ available  for $k=1,2$. While such Teichm\"uller curves are not moving in $\ov{\mathbb{H}}^k_{g, 1}$, their union is indeed dense in $\ov{\mathbb{H}}^k_{g, 1}$.

\begin{proof}[Proof of Theorem \ref{thm:Z1extremal}]
First, we  show the extremality for $k=1$. Let $\mathcal{C}\subset \P\E^1_{g,1}$ be the closure of a Teichm\"uller curve generated by some element $(C, P, \mu)$ 
in $\mathbb{H}^1_{g, 1}$ such that $C$ is smooth and $\mu$ is a canonical divisor on $C$ supported at distinct points, including $P$.
By varying $(C, P, \mu)$,
the collection $\mathscr{C}$ of such Teichm\"uller curves $\mathcal{C}$ 
is dense in $\ov{\mathbb{H}}^1_{g, 1}$.
From \cite{MR3033521}, one has for each $\mathcal{C}\in \mathscr{C}$:
\begin{align}
\begin{split}
\label{eq:intersectionCTeich}
&\mathcal{C} \cdot \lambda = -\frac{\chi}{2}L, \qquad\qquad\qquad\qquad\qquad\,\,
\mathcal{C} \cdot \delta_0 = -\frac{3\,\chi}{2}(4L - g+1), \\
&\mathcal{C} \cdot \psi = 2\, \frac{\mathcal{C} \cdot \lambda - \mathcal{C} \cdot \delta_0/12}{g-1} = -\frac{\chi}{4}, \qquad
\mathcal{C} \cdot \delta_i = 0, \quad\mbox{ for $i \geq 1$,}
\end{split}
\end{align}
where $\chi$ is the Euler characteristic of $\mathcal{C}$ (a different sign convention is used for $\chi$ in \cite{MR3033521}), and $L$ is the sum of the top $g$ Lyapunov exponents of~$\mathcal{C}$ (we refer to \cite[\S 1.4]{eskin2014sum} for more on Lyapunov exponents). 
Specifically, $\mathcal{C}$ is disjoint from $\delta_i$ for $i>0$ from \cite[Cor.~3.2]{MR3033521}.
For any Teichm\"uller curve, the intersections with $\lambda$ and $\delta_0$  are computed in terms of the area Siegel-Veech constant $c_{\rm{SV}}(\mathcal{C})$ of $\mathcal{C}$ \cite[\S 1.6]{eskin2014sum}: namely,
  $\deg \delta_0 = -6\chi\, c_{\rm{SV}}(\mathcal{C})$ from \cite[Proof of Prop.~4.5]{MR3033521} and $\deg \lambda = \frac{L}{12\, c_{\rm{SV}}(\mathcal{C})} \deg \delta_0$ from \cite[Prop.~4.5]{MR3033521}.
Since $\pi(\mathcal{C})$ is generically contained in the principal stratum of $\P\E^1_g$ consisting of curves with canonical divisors supported at \textit{distinct} points, where $\pi\colon \P\E^1_{g,1}\rightarrow \P\E^1_g$ is the forgetful map, one has $c_{\rm{SV}}(\mathcal{C})=L-\frac{1}{4}(g-1)$ from \cites[Thm.~1]{eskin2014sum}[Thm.~4.4]{MR3033521}, hence the intersections with $\lambda$ and $\delta_0$ in \eqref{eq:intersectionCTeich} follow. The intersection with $\psi$ follows from \cite[Prop.~4.8]{MR3033521} using that for the generator $(C,P,\mu)$ of $\mathcal{C}$ the support of the canonical divisor $\mu$ contains  $P$ with multiplicity one.

Moreover, from \cites{moller2006variations}[(15)]{MR3033521}, one has
\[
\mathcal{C} \cdot \eta = -\frac{\chi}{2}.
\]
Combining this with Lemma \ref{lem:Z1class}, we have  
\[
\mathcal{C} \cdot \ov{\mathbb{H}}^1_{g, 1} = \mathcal{C} \cdot (\psi - \eta) = \frac{\chi}{4}.
\] 

Since $\psi$ is ample on any nonconstant family not all of whose elements are singular \cite[Thm.~6.33]{MR1631825}, one has $\mathcal{C} \cdot \psi>0$.   
We deduce  $\chi < 0$ and thus $\mathcal{C} \cdot \ov{\mathbb{H}}^1_{g, 1} <0$. We conclude that $\mathcal{C}\subset \ov{\mathbb{H}}^1_{g, 1}$ for all $\mathcal{C}\in \mathscr{C}$.

To deduce the extremality of $\ov{\mathbb{H}}^1_{g, 1}$, we apply Lemma \ref{lem:extremality} for a certain big divisor class $B$ on $\P\E^1_{g,1}$.
Select an ample divisor class 
 \[ 
 A := c_\eta \,\eta + c_\lambda \,\lambda + c_\psi \,\psi + \sum_{i=0}^{g-1} c_i\,\delta_i \,\,\in\,\, \mathrm{Pic}\left(\P\E^1_{g,1}\right),
 \] 
and define $B:=dA$ for a sufficiently small  $d>0$ such that $\mathcal{C} \cdot (\ov{\mathbb{H}}^1_{g, 1} + B) \leq 0$ for all $\mathcal{C} \in \mathscr{C}$. Namely, let 
 \[ 
 d := \inf_{\mathcal{C}\in \mathscr{C}} \Bigg\{ \frac{1}{2c_\eta + c_\psi -6(g-1)c_0 +2L(c_\lambda +12c_0)} \Bigg\}. \] 
 The expression in the brackets comes from solving for $d$ in $\mathcal{C} \cdot (\ov{\mathbb{H}}^1_{g, 1} + dA) = 0$ for a given Teichm\"uller curve, and it is positive since $\mathcal{C} \cdot \ov{\mathbb{H}}^1_{g, 1} <0$ and $\mathcal{C} \cdot A >0$. Furthermore, the value in the brackets depends only on the sum $L$ of the top $g$ Lyapunov exponents of $\mathcal{C}$. Since  $0\leq L\leq g$ (as the top $g$ Lyapunov exponents are nonnegative and at most $1$ \cites[\S 5]{kontsevich1997lyapunov}[\S 1.4]{eskin2014sum}), the infimum here is indeed  positive, and thus $B=dA$ is big. The extremality of $\ov{\mathbb{H}}^1_{g, 1}$ follows by  Lemma \ref{lem:extremality}.
 
\smallskip 
 
Next, we show the extremality for  $k=2$. The argument is similar to the previous case. Let $\mathcal{C}\subset \P\E^2_{g,1}$ be the closure of a Teichm\"uller curve generated by some  $(C, P, \mu)$ in $\mathbb{H}^2_{g, 1}$ such that $C$ is smooth and $\mu$ is the class of a quadratic differential on $C$ vanishing at distinct points, including $P$, and $\mu$ is not the class of the square of an Abelian differential \cite[\S 4.1]{chen2014quadratic}. The collection $\mathscr{C}$ of such Teichm\"uller curves $\mathcal{C}$ is dense in $\ov{\mathbb{H}}^2_{g, 1}$.
From \cite[Prop.~4.2 and (2)]{chen2014quadratic}, one has
\begin{align}
\label{eq:intersectionCTeich2}
\mathcal{C} \cdot \lambda = -\frac{\chi}{36}\big(18\, c_{\rm{SV}}(\mathcal{C}) + 5(g-1)\big), \quad
\mathcal{C} \cdot \delta = -6\,\chi \, c_{\rm{SV}}(\mathcal{C}), \quad
\mathcal{C} \cdot \psi = -\frac{\chi}{3}, 
\end{align}
where, as before, $\chi$ and $c_{\rm{SV}}(\mathcal{C})$ are the Euler characteristic and the area Siegel-Veech constant of $\mathcal{C}$, respectively (a different sign convention is used for $\chi$ in \cite{chen2014quadratic}),
 and $\delta=\delta_0+\cdots + \delta_{g-1}$ is the total boundary divisor class. From \cites{moller2006variations}[Proof of Prop.~4.2]{chen2014quadratic}, one  has $\mathcal{C} \cdot \eta = -\chi$. By Lemma \ref{lem:Z1class}, we conclude that 
\[ 
\mathcal{C} \cdot \ov{\mathbb{H}}^2_{g, 1} = \mathcal{C} \cdot (2\psi-\eta) = \frac{\chi}{3}. 
\] 
As in the case $k=1$, one argues  that $\chi<0$, hence $\mathcal{C} \cdot \ov{\mathbb{H}}^2_{g, 1} < 0$. Now, select an ample divisor class
\[
A:= c_\eta \eta + c_\lambda \lambda + c_\psi \psi + \sum_{i=0}^{ g-1 } c_i \delta_i \,\,\in\,\, \mathrm{Pic}\left(\P\E_{g,1}^2\right),
\] 
 and let $c_\delta := \max_{\, i=0}^{\, g-1} c_i$. Then $A^+=  c_\eta \eta + c_\lambda \lambda + c_\psi \psi + c_\delta \delta$ is big.
Define $B:=dA^+$ for some sufficiently small  $d>0$ such that $\mathcal{C} \cdot (\ov{\mathbb{H}}^2_{g, 1} +B) \leq 0$ for all $\mathcal{C} \in \mathscr{C}$. Specifically, let 
\[ 
d:= \inf_{\mathcal{C}\in \mathscr{C}} \Bigg\{ \frac{12}{36c_\eta + 12c_\psi + 5(g-1)c_\lambda +c_{\rm{SV}}(\mathcal{C})(18c_\lambda +216c_\delta)} \Bigg\}.
\] 
The expression in the brackets is the value of $d$ such that $\mathcal{C} \cdot (\ov{\mathbb{H}}^2_{g, 1} +dA^+) = 0$, and is positive, since $\mathcal{C} \cdot \ov{\mathbb{H}}^2_{g, 1} < 0$, $\mathcal{C} \cdot A > 0$, and $\mathcal{C} \cdot \delta_i \geq 0$ for all~$i$ (as the general element of $\mathcal{C}$ is a \textit{smooth} pointed curve with quadratic differential), hence $\mathcal{C} \cdot A^+ > 0$.
From  \cite[Thm.~2]{eskin2014sum}, $c_{\rm{SV}}(\mathcal{C})$ is the sum of certain Lyapunov exponents minus a positive constant; in particular,  $c_{\rm{SV}}(\mathcal{C})$ is bounded. It follows that $d$ is positive, and thus $B$ is big. By Lemma \ref{lem:extremality}, $\ov{\mathbb{H}}^2_{g, 1}$ is extremal.

\smallskip

Finally, we prove the  rigidity of $\ov{\mathbb{H}}^k_{g, 1}$ simultaneously for $k\in \{1,2\}$. Suppose that  there exists an effective divisor $D$ such that $D \equiv m\,\ov{\mathbb{H}}^k_{g, 1}$, for some $m>0$. We may assume that $D$ does not contain $\ov{\mathbb{H}}^k_{g, 1}$, since otherwise we could simply consider the divisor $D\setminus\ov{\mathbb{H}}^k_{g, 1}$ and reduce the coefficient $m$. For all Teichm\"uller curves $\mathcal{C}\in\mathscr{C}$ as above, we have that $\mathcal{C} \cdot \ov{\mathbb{H}}^k_{g, 1} < 0$, and thus $\mathcal{C} \cdot D < 0$. This means that $D$ contains the entire collection $\mathscr{C}$ of such curves. Since $\mathscr{C}$ is dense in $\ov{\mathbb{H}}^k_{g, 1}$, we deduce that $D$ contains $\ov{\mathbb{H}}^k_{g, 1}$, which contradicts the assumption. 
\end{proof}

\bibliographystyle{alphanumN}
\bibliography{Biblio}

\end{document}